\documentclass[12pt]{article}
\usepackage[cp1251]{inputenc}
\usepackage[usenames]{color}
\usepackage{colortbl}

\usepackage{amssymb, amsthm, amsmath, amsfonts}
\usepackage{graphicx,amscd}
\usepackage{caption}
\usepackage{subcaption}
\usepackage{amsopn}
\usepackage[T2A]{fontenc}
\usepackage{float}
\usepackage{enumitem}

\newtheorem{theorem}{Theorem}

\newtheorem{corollary}{Corollary}

\newtheorem{statement}{Statement}[section]

\newcommand{\Ind}{\mathop{\rm Ind}\nolimits}

\begin{document}
{\bf On isolated periodic points of diffeomorphisms with expanding attractors of codimension 1}

{Marina Barinova, HSE University}

\begin{abstract} 
In the paper we consider an $\Omega$-stable 3-diffeomorphism, chain recurrent set of which consists of isolated periodic points and expanding attractors of codimension 1, orientable or not. We estimate a minimum number of isolated periodic points using information about the structure of the attractors.
\end{abstract}

\section{Introduction and formulation of results}\label{sec:introduction}

Let $M^n$ be a closed smooth connected $n$-manifold with a metric $d$ and $f:M^n\to M^n$ be a diffeomorphism. An invariant compact set $\Lambda \subset M^n$ is called  \textit{hyperbolic} if there is a continuous $Df$-invariant splitting of the tangent bundle $T_\Lambda M^n$ into {\em stable} and {\em unstable subbundles}
$E^s_\Lambda\oplus E^u_\Lambda$, $\dim E^s_x + \dim E^u_x = n$ ($x\in \Lambda$) such that for natural $k$ and for some fixed $C_s>0$, $C_u>0$, $0<\lambda <1$
\[\begin{array}{ll}
    \Vert Df^k(v)\Vert \leq C_s\lambda ^k\Vert v\Vert, & \quad v\in E^s_{\Lambda},\\
    \Vert Df^{-k}(w)\Vert \leq C_u\lambda ^k\Vert w\Vert, & \quad w\in E^u_{\Lambda}.
\end{array}\]

Recall that {\it $\varepsilon$-chain of the length $m\in\mathbb N$}, joining points $x,y\in M^n$, for $f$ is called a collection of points $x=x_0,\dots,x_m=y$ such that  $d(f(x_{i-1}),x_{i})<\varepsilon$ for $1\leqslant i\leqslant m$. A point $x\in M^n$ is called {\it chain recurrent} for $f$ if for any $\varepsilon>0$ there exists $m$, depending on $\varepsilon>0$, and an $\varepsilon$-chain of length $m$, joining $x$ to itself. The set of all chain recurrent points is called {\it chain recurrent set} and is denoted by $\mathcal R_f$. 

Summarizing the results in \cite{PaMelo}, \cite{ShubStab}, \cite{SmaleOmega}, \cite{FrankeHyp} we know that the  hyperbolicity of $\mathcal R_f$  is equivalent to  $\Omega$-stability of $f$, that is small perturbations of $f$ preserve the chain recurrent set (equivalently non-wandering set $NW(f)$) structure. Thus, by \cite{Smale1967}, $\mathcal R_f$ consists of a finite number of pairwise disjoint sets, called \textit{basic sets}, each of which is compact, invariant, and topologically transitive (contains a dense orbit). If a basic set is a periodic orbit, then it is named \textit{trivial}. In the opposite case, it is \textit{non-trivial}. 
If $\dim\Lambda=n-1$ for some basic set $\Lambda$ then it is called a \textit{basic set of co-dimension 1}. 

A stable and unstable manifolds of a point $x\in\Lambda$, where $\Lambda$ is a basic set, can be defined in the following way:
\[\begin{array}{ll}
W^s_x=\{y\in M^n\mid \lim\limits_{k\to +\infty} d(f^k(x),f^k(y))=0\},\\
W^u_x=\{y\in M^n\mid \lim\limits_{k\to +\infty} d(f^{-k}(x),f^{-k}(y))=0\}.
\end{array}\]
By \cite{Smale1967}, $W^s_x$ and $W^u_x$ are injective immersions of $\mathbb{R}^q$ and $\mathbb{R}^{n-q}$, accordingly, for some $q\in\{0,1,\ldots,n\}$. For $r>0$ we denote by $W^{s}_{x,r}$ and $W^{u}_{x,r}$  the immersions of discs $D^q_r\subset\mathbb R^q$ and $D_r^{n-q}\subset\mathbb{R}^{n-q}$.
 
The concept of orientability can be introduced for a basic set $\Lambda$ with $\dim W^{s}_x=1$ or $\dim W^{u}_x=1$, $x\in\Lambda$. A non-trivial basic set $\Lambda$ is called {\it orientable} if for any point $x\in\Lambda$ and any fixed numbers $\alpha > 0,\,\beta> 0$ the intersection index\footnote{Let $J^k:\mathbb{R}^k\rightarrow{M}^3$ be immersions, $D^k$ be open balls of finite radii in $\mathbb{R}^k$, $k=1,2$.
Then the restrictions $J^k:D^k\rightarrow{M}$ are embeddings and their images $W^k=J^k(D^k)$ are smooth embedded submanifolds of the manifold $M^3$. Let $U^k$ be a  tubular neighborhood of $W^k$, which are images of embeddings in $M^3$ of spaces of $(3-k)$-dimensional vector bundles on $W^k$ \cite[Chapter~4, par.~5]{Hirsch}.
Since the balls $D^k$ are contractible, then these bundles are trivial and, hence,  $U^2\setminus{W}^2$ consists of two connected components $U^2_+$ and $U^2_-$. It allows to define a function $\sigma:U^2_+\cup U^2_-\rightarrow\mathbb{Z}$, such that $\sigma(x)=1$ if $x\in U^2_+$ and $\sigma(x)=0$ if $x\in U^2_-$. If  submanifolds $W^1$ and $W^2$ intersect transversally at a point $x=J^1(t)$, $t\in D^1$, then there exists a number $\delta>0$ such that $J^1((t-2\delta,t+2\delta))\subset U^2$. The number
\[\Ind_x(W^1,W^2)=\sigma(t+\delta)-\sigma(t-\delta)\]
 is called an {\it intersection index} of submanifolds $W^1$ and $W^2$ in the point $x$. Notice, that this definition does not require orientability of the manifold $M^3$.} $W^u_{x,\alpha}\cap W^s_{x,\beta}$ is the same at all intersection points
($+1$ or $-1$) \cite{Grines1975}. Otherwise, the basic set is called {\it non-orientable}.

A basic set $\Lambda$ is called an if it has a compact {\it trapping neighborhood} $U$, such that $f(U)\subset int~U$ and $\bigcap\limits_{n=1}^{+\infty}f^n(U)=\Lambda$. Each hyperbolic attractor consists of unstable manifolds of its points by \cite{Plykin1971}. If $\dim\Lambda=\dim W^u_x$, $x\in\Lambda$, for a hyperbolic attractor $\Lambda$, then it is \textit{expanding}.

Any co-dimension 1 expanding attractor $\Lambda$ divides its basin $W^s_\Lambda$ into a finite number of connected components. Every such a component $B$ determines 
{\it a bunch $b$} as the union of unstable manifolds of all periodic points from $\Lambda$ whose stable separatrix belongs to $B$. The number $k$ of such so-called \textit{boundary points} is finite and it is called {\it a degree of the bunch $b$} and $b$ is called {\it $k$-bunch} with the {\it basin} $B$. 

If  $n\geqslant 3$ then, by \cite{Plykin1984}[Theorem 2.1], any co-dimension 1 expanding attractor $\Lambda$  has 1-bunches or 2-bunches only. Moreover, the following fact takes place. 

\begin{statement}\label{prop:attractor_structure} If $\Lambda$ is a hyperbolic expanding attractor of co-dimension 1 of a diffeomorphism $f:M^n\to M^n$ given on a closed smooth $n$-manifold $M^n$, then $\Lambda$ is non-orientable iff it has an 1-bunch. 
\end{statement}

In the paper we consider diffeomorphisms every non-trivial basic set of which is an expanding attractor of codimension 1, and investigate the properties of such diffeomorphisms and the structure of their ambient manifolds. 
 The main result is the following theorem.

\begin{theorem}\label{theo:minimum_number}
Let $f:M^3\to M^3$ be an $\Omega$-stable diffeomorphism, given on a closed 3-manifold, $\Lambda$ be a non-empty set of non-trivial basic sets of $f$. If $\Lambda$ consists of expanding attractors of codimension 1 having a total of $k_1$ bunches of degree 1 and $k_2$ bunches of degree 2, then the number of points in the set $NW(f)\setminus\Lambda$ no less then $\frac32 k_1+k_2$ and this estimate is exact.
\end{theorem}

\begin{corollary}\label{cor:3-sphere} If the non-wandering set $NW(f)$ of an $\Omega$-stable diffeomorphism $f:M^3\to M^3$ consists of 2-dimensional expanding attractors with $k$ bunches in total and $k$ isolated periodic points, then 
\begin{itemize}
\item each non-trivial attractor and $M^3$ are orientable;
\item $\dim W^u_p=1$ for every isolated saddle point $p$;
\item each connected component of the set $M^3\setminus\Lambda$ is homeomorphic to a punctured 3-sphere. 
\end{itemize}
\end{corollary}

It is clear from сorollary \ref{cor:3-sphere}, that in a subclass of diffeomorphism with orientable $\Lambda$ and non-orientable $M^3$ the estimates from Theorem \ref{theo:minimum_number} can not be reached. For this case the following theorem takes place.

\begin{theorem}\label{theo:non-or_manifold}
Let an $\Omega$-stable diffeomorphism $f:M^3\to M^3$ be given on closed non-orientable manifold $M^3$ and a set of non-trivial basic sets consists of expanding orientable 2-dimensional attractors having a total of $k$ bunches, then the number of isolated periodic points is no less than $k+2$.
\end{theorem}

A simple structure of the orbit space of the restriction of $f$ to the set $W^s_\Lambda\setminus\Lambda$ gives us a way to obtain an $\Omega$-stable system without non-trivial basic sets from considered one. We will describe a procedure of transition from a cascade with codimension 1 expanding attractors to a corresponding regular system in a section \ref{sec:transition_to_regular_system}. 
A section \ref{sec:proof_of_theo} gives a proof of estimates from theorem \ref{theo:minimum_number} and theorem \ref{theo:non-or_manifold}. A proof of corollary \ref{cor:3-sphere} is directly follows from the proof of theorem \ref{theo:minimum_number}. Finally, in the section \ref{sec:realization} we show that estimates are exact.

\section{Transition to a regular system}\label{sec:transition_to_regular_system}

In this section we will show how to obtain a system $\widetilde f:\widetilde M^3\to\widetilde M^3$ with regular dynamics from a system $f:M^3\to M^3$ with codimension 1 expanding attractors and isolated periodic points. 

\begin{figure}[h]
\centering
\begin{subfigure}{0.48\textwidth}
    \includegraphics[width=\textwidth]{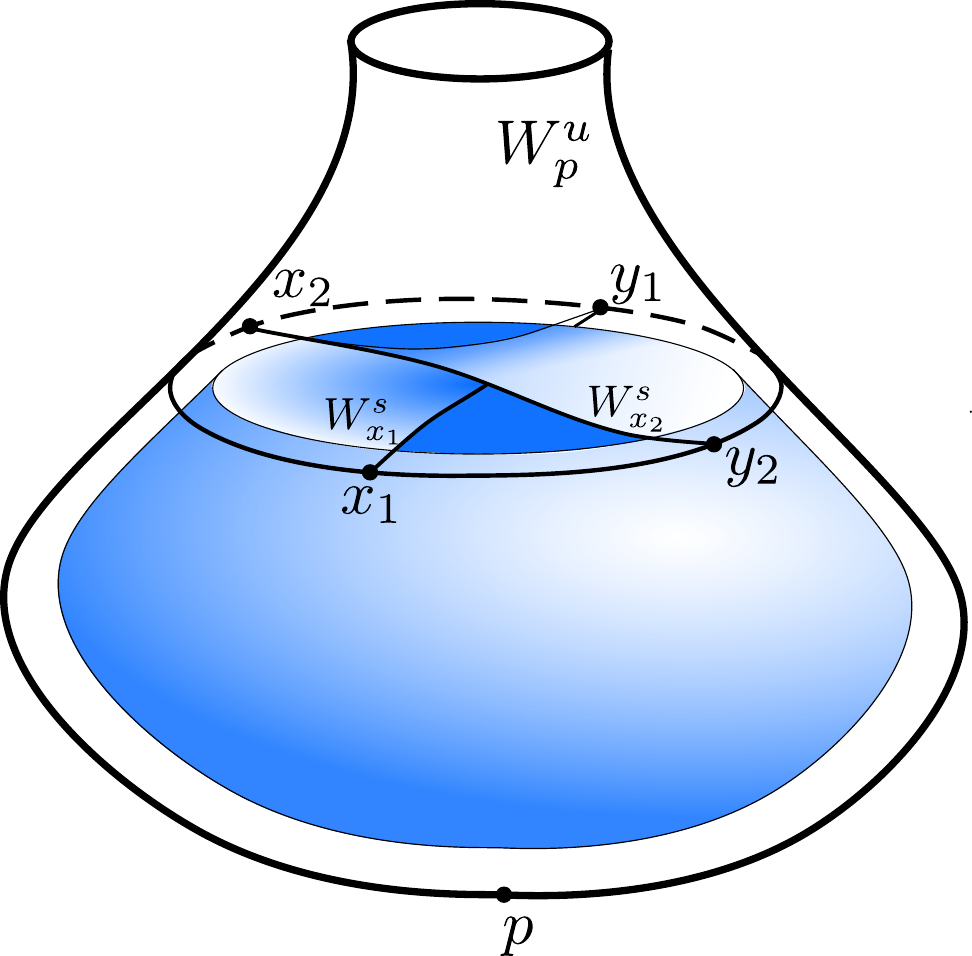}
    \caption{Bunch of degree 1}
    \label{fig:first}
\end{subfigure}
\hfill
\begin{subfigure}{0.48\textwidth}
    \includegraphics[width=\textwidth]{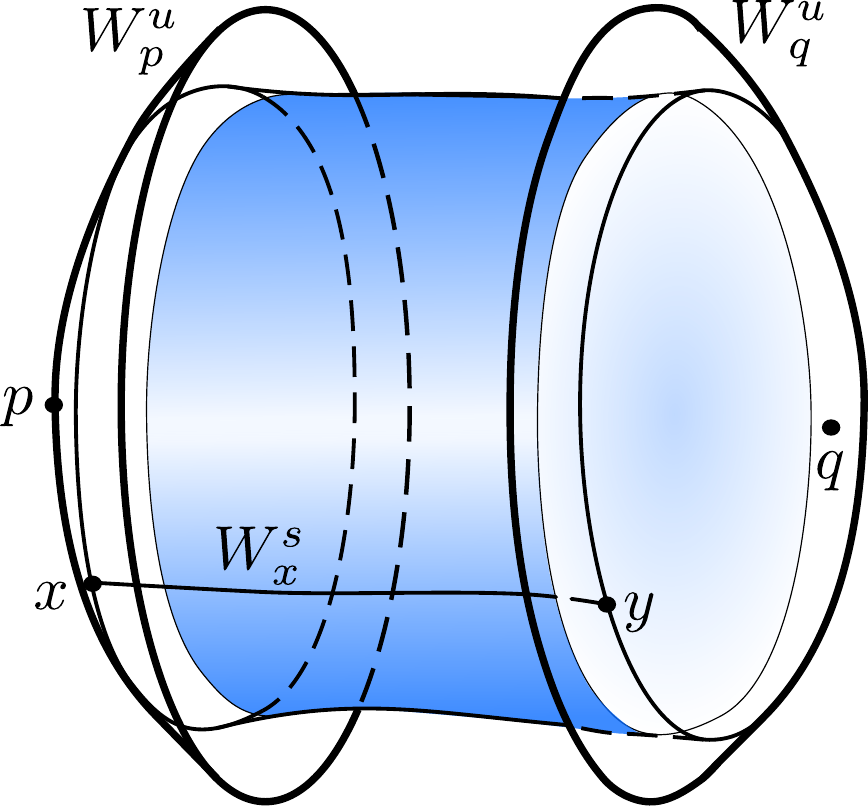}
    \caption{Bunch of degree 2}
    \label{fig:second}
\end{subfigure}
       
\caption{Components of the boundary of a trapping neighborhood near bunches of different degrees}
\label{fig:bunches}
\end{figure}

Let $\Lambda$ be a set of non-trivial attractors of $f$ and $U_\Lambda$ be its trapping neighborhood. The boundary of $U_\Lambda$ consists of $k_1$ copies of $\mathbb RP^2$ and $k_2$ copies of $\mathbb S^2$ (\cite{BaGrPo2023}[lemma 2.2]) as in figure \ref{fig:bunches}. 
Let  $M^3\setminus 
int\,U_\Lambda=M^+\sqcup M^-$, where $M^+$ and $M^-$ are compact subsets of $M^3$ (one of them can be empty) such that $\partial M^+$ consists of $k^+$ 2-spheres, $\partial M^-$ consists of $k^-_1>0$ copies of $\mathbb RP^2$ and $k^-_2$ copies of $\mathbb S^2$. Notice, that each connected component of $M^-$ is non-orientable \cite{MeZhu2002} and hence there exists a double cover $\pi:\widehat M^-\to M^-$ \cite{BaPoYa2024}, such that $\partial \widehat M^-$ consists of $\hat k^-=k^-_1+2k^-_2$ 2-spheres. There is the following division of $\widetilde M^3$ on disjoint closed submanifolds $\widetilde M^+$ and $\widetilde M^-$:

\begin{itemize}
\item $\widetilde M^+=M^+\cup_{h^+} (D\times\mathbb Z_{k^+})$, where $D=\{(x,y,z)\in\mathbb R^3\mid x^2+y^2+z^2\leqslant 1\}$, $h^+:\partial M^+\to\partial (D\times\mathbb Z_{k^+})$ is a diffeomorphism;

\item $\widetilde M^-=\widehat M^-\cup_{h^-} (D\times\mathbb Z_{\hat k^-})$, $h^-:\partial \widehat M^-\to\partial (D\times \mathbb Z_{\hat k^-})$ is a diffeomorphism.  
\end{itemize}

Let us introduce the following designations: 
\begin{itemize}
\item $\mathcal M^+=\bigcup\limits_{m=1}^{+\infty}f^m(M^+)$, $\mathcal M^-=\bigcup\limits_{m=1}^{+\infty}f^m(M^-)$;
\item $\pi:\widehat{\mathcal M}^-\to{\mathcal M}^-$ is a double cover of $\mathcal M^-$;
\item $\widehat{\mathcal M}={\mathcal M}^+\cup\mathcal {\widehat M}^-$, $k=k^++\hat k^-$;
\item $\widehat f:\widehat{\mathcal M}\to\widehat{\mathcal M}$ is a diffeomorphism such that $\widehat f|_{\mathcal M^+}=f|_{\mathcal M^+}$ and $\widehat f|_{{\mathcal M}^-}$ is a lift of $f|_{{\mathcal M}^-}$.
\end{itemize}
Let also $O$ be the centre of the disk $D$.

\begin{theorem}
There exists a diffeomorphism $\widetilde f:\widetilde M^3\to\widetilde M^3$, which has $k$ sinks at the points $O\times \mathbb Z_{k}\subset\widetilde M^3$ and $\widetilde f|_{\widetilde M^3\setminus (O\times \mathbb Z_{k})}$ is topologically conjugated with $\widehat f$.
\end{theorem}

\begin{proof}

Let ${\mathcal B}^+$ and ${\mathcal B}^-$ be sets of the bunch basins in the sets $\mathcal M^+$ and $\mathcal M^-$ correspondingly. Let also $\widehat {\mathcal B}^-={\pi}^{-1}({\mathcal B}^-)$, and $\widehat {\mathcal B} = {\mathcal B}^+\cup \widehat {\mathcal B}^-$. Since bunch basins are periodic, then there exists a division of the set $\widehat {\mathcal B}$ on subsets $\widehat {\mathcal B}_i$, $i=1,\ldots,l$, each of which has a minimum natural number $m_i$ such that the set $\widehat{\mathcal B}_i=\bigcup\limits_{j=1}^{m_i}f^{j}(\widehat B_i)$, where $\widehat B_i$ is some connected component of $\widehat {\mathcal B}$. Then $m_1+\cdots+m_l=k$. It follows from \cite{BaGrPo2023} that each  $\widehat B_i$ is diffeomorphic to $\mathbb S^2\times \mathbb R$ and hence the orbit space of $f|_{\widehat {\mathcal B}_i}$ is diffeomorphic to $\mathbb S^2 \times\mathbb S^1$, if $\widehat f^{m_i}|_{\widehat {B}_i}$ preserves orientation, or $\mathbb S^2\widetilde\times S^1$, if $\widehat f^{m_i}|_{\widehat {B}_i}$ reverses one. Notice, that periodic hyperbolic sinks have the same orbit spaces in their basins. 

Let $g_{i}:\mathbb R^3\times\mathbb Z_{m_i}\to\mathbb R^3\times\mathbb Z_{m_i}$ be a diffeomorphism with $m_i$ sinks at the origins $O\times\mathbb Z_{m_i}$, $g_i=(\frac x2,\,\frac y2,\,\frac z2,\,t+1\mod m_i)$, if $\widehat f^{m_i}|_{\widehat {\mathcal B}_i}$ preserves orientation, and $g_i=(-\frac x2,\,\frac y2,\,\frac z2,\,t+1\mod m_i)$ otherwise. Let also $h_{i}:\widehat{\mathcal B}_{i}\to(\mathbb R^3\setminus O)\times \mathbb Z_{m_i}$ be diffeomorphisms, conjugated $\widehat f|_{\widehat{\mathcal B}_{i}}$ with $g_{i}|_{(\mathbb R^3\setminus O)\times \mathbb Z_{m_i}}$. Then diffeomorphisms $g:\mathbb R^3\times\mathbb Z_{k}\to\mathbb R^3\times\mathbb Z_{k}$ and $h:\widehat B\to(\mathbb R^3\setminus O)\times\mathbb Z_{k}$ can be composed of $g_i$ and $h_{i}$. Moreover, $h$ can be chosen in such a way that $h(U_\Lambda)=\mathbb S^2\times\mathbb Z_{k}$, where $\mathbb S^2\subset\mathbb R^3$ is a standard 2-sphere. Then $\widetilde M^3 = \widehat{\mathcal M}\sqcup_{h} (\mathbb R^3\times\mathbb Z_{k})$ with a natural projection $q:\widehat {\mathcal M} \sqcup(\mathbb R^3\times\mathbb Z_{k})\to\widetilde M^3$. The desired diffeomorphism $\widetilde f$ coincides with $q \widehat f(q|_{\widehat {\mathcal M}})^{-1}$ on the set $q(\widehat {\mathcal M})$ and with $q g(q|_{\mathbb R^3\times\mathbb Z_{k}})^{-1}$ on the sets $q(\mathbb R^3\times\mathbb Z_{k})$.
Notice, that by the construction $\widetilde f$ has $k$ sinks more then $\widehat f$.
\end{proof}

\section{Low estimate of trivial basic sets number}\label{sec:proof_of_theo}
In this section we will prove the estimate from theorem \ref{theo:minimum_number}. 
Let $f:M^3\to M^3$ be an $\Omega$-stable diffeomorphism, given on a closed connected 3-manifold. 
Everywhere below in this section we will assume that all isolated periodic points and also boundary periodic points are fixed, because it does not affect the lower estimates: an appropriate degree of initial system satisfies these property and has the same number of isolated periodic points. Let $R_f=\Lambda\cup p_1\cup p_2\cup\ldots\cup p_m$, where $\Lambda$ is a union of expanding attractors of codimension 1 with $k_1$ bunches of degree 1 and $k_2$ bunches of degree 2 in total and $p_i$ is a fixed point, $i\in\{1,2,\ldots,m\}$. Below we will prove that $m\geqslant \frac32 k_1+k_2$.

\begin{proof}

Via the transition to a regular system, described in section \ref{sec:transition_to_regular_system}, we will obtain an $\Omega$-stable diffeomorphism $\widetilde f:\widetilde M^3\to\widetilde M^3$ with a finite chain-recurrent set on a closed manifold $\widetilde M^3$. Notice, that all chain-recurrent points of $\widetilde f$ are fixed.
Let $M$ be a connected component of $\mathcal M=M^3\setminus\Lambda$. Notice that $M$ is $f$-invariant. There exists a connected component $\widetilde M$ of $\widetilde M^3$ corresponded to $M$.  

Let us denote a number of 1- and 2-bunch basins, contained in $M$, as $l_{1}$ and $l_{2}$ correspondingly. $M$, and hence $\widetilde M$, can be one of 2 types (see section \ref{sec:transition_to_regular_system}): (1) $M\subset {\mathcal M}^+$ and (2) $M\subset {\mathcal M}^-$. In the first case $l_1=0$ and $\widetilde f|_{\widetilde M}$ has $l_2$ sinks more than $\widehat f|_{M}$. In the second case $l_{1}>0$ and even and $\widetilde f|_{\widetilde M}$ has $l_1+2l_2$ sinks more than $\widehat f|_{{\pi}^{-1}(M)}$.

Let $C_{j}$, $j=0,1,2,3$, be a number of fixed points $p$ of $\widetilde f|_{\widetilde M}$ with $\dim W^u_p=j$, for example, $C_{0}$ be a number of sinks.  
Also $\widetilde f|_{\widetilde M}$ has at least 1 source, since it is $\Omega$-stable. Then by the Lefschetz formula the alternating sum of $C_{j}$ is equal to 0: \[C_{3}-C_{2}+C_{1}-C_{0}=0.\]
At the same time since $\widetilde M$ is connected, then $C_{1}-C_{0}+1\geqslant 0$ \cite{GrLauPo2009}. If $\widetilde M$ of the type (1) then $C_{0}\geqslant l_2>0$ and there is no additional restrictions. The finding of the minimum of the sum $C_{0}+C_{1}+C_{2}+C_{3}$ is a linear programming problem, it can be solved by a simplex method. Then the minimum of fixed points of $\widetilde f|_{\widetilde M}$ can be reached if $C_{3}=1$, $C_{2}=0$, $C_{1}=l_{2}-1$, $C_{0}=l_{2}$. Therefore $f|_{M}$ has at least $l_{2}$ isolated fixed points if $\widetilde M$ of the type (1). It follows from \cite{OsPo2024} that $\widetilde M$ is homeomorphic to $S^3$ in this case.

If $\widetilde M$ of the type (2) then $C_{1}$, $C_{2}$, and $C_{3}$ are even, because isolated periodic points of $f|_{M}$ is doubled in this case. Also $C_{0}\geqslant l_1+2l_{2}>0$. Without loss of generality we suppose that 1-dimensional se\-pa\-rat\-rices of saddles do not intersect\footnote{Each $\Omega$-stable diffeomorphism with finite chain-recurrent set has $\varepsilon$-close Morse-Smale diffeomorphism with the same amount of chain-recurrent points. Therefore we can consider this Morse-Smale diffeomorphism instead of initial one to calculate desired estimates.}. Then we can arrange points in the non-wandering set of $\widetilde f|_{\widetilde M}$ agreed with Smale relation\footnote{Let $\Lambda_1$ and $\Lambda_2$ be basic sets of an $\Omega$-stable diffeomorphism $f:M\to M$. $\Lambda_1\prec\Lambda_2$ if $W^s_{\Lambda_1}\cap W^u_{\Lambda_2}\neq\varnothing$.}. Moreover, the order can be chosen in such a way that each saddle of index 1 comes before all saddles of index 2. Thus we have $\omega_{1}\prec\ldots\prec\omega_{_{C_0}}\prec\sigma_{1}\prec\ldots\prec\sigma_{_{C_1}}\prec\beta_{1}\prec\ldots\prec\beta_{_{C_2}}\prec\alpha_{1}\prec\ldots\prec\alpha_{_{C_3}}$, where each $\omega_{i}$ is a sink, each $\sigma_{i}$ is a saddle of index 1, each $\beta_{i}$ is a saddle of index 2, and each $\alpha_{i}$ is a source. %Let also $\sigma_{i,2s-1}$ and $\sigma_{i,2s}$ correspond to the same saddle of $f$.

It follows from the paper \cite{GMPZ-global} that a set $\mathcal A = \bigcup\limits_{i=1}^{{c_1}} cl(W^u_{\sigma_{i}})$ is 1-dimensional and connected. The double cover $\pi$ induces an involution $\varphi$ on the set $\mathcal A\setminus{(\omega_{1}\cup\ldots\cup\omega_{_{C_0}})}$ and can be extended by continuity on the whole $\mathcal A$. Moreover, a set of fixed points of the extended involution $\varphi$ coincides with the set of sinks corresponded to 1-bunches.

Let $\mathcal A^*=\mathcal A /_\varphi$. Since a natural projection is a continuous map, then connectedness of $\mathcal A$ implies the connectedness of $\mathcal A^*$. $\mathcal A^*$ contains $(C_0+l_1)/2$ sinks and hence it is needed at least $(C_0+l_1)/2-1$ saddles of index 1. Therefore $\mathcal A$ contains at least $(C_0+l_1-2)$ saddles of index 1, i.e. $C_1\geqslant C_0+l_1 -2$.

Let us solve a linear programming task for this case:

\begin{equation*}
\begin{array}{l}
C_{3}-C_{2}+C_{1}-C_{0}=0, \\
C_{0}\geqslant l_1+2l_2, \\
C_1-C_0\geqslant l_1 -2, \\
C_3\geqslant 2.
\end{array}
\end{equation*}

The optimal values are: $C_0=l_1+2l_2$, $C_1=2l_1+2l_2-2$, $C_2=l_1$, and $C_3=2$.
Then there are at least $l_1+l_2-1$ saddles of index 1, $l_1/2$ saddles of index 2, and 1 source at the component $M$. 

Summing over all connected components of $\mathcal M$ we obtain that $f$ has at least $\frac32 k_1 + k_2$ isolated periodic points: at least $s$ sources, $(k_1 + k_2 - s)$ saddles of index 1, and $k_1/2$ saddles of index 2, where $s$ is a number of connected components of $\mathcal M$.
\end{proof}

Below we will prove theorem \ref{theo:non-or_manifold}.

\begin{proof}

If $M^3$ is non-orientable, but $\Lambda$ contains only orientable attractors, then by \cite{Plykin1984} $W^s_{\Lambda}$ is homeomorphic to a punctured 3-torus, $\mathcal M^-=\varnothing$, and there exists a non-orientable connected component $M$ of the set $\mathcal M^+$. Then corresponded manifold $\widetilde M$ is also non-orientable, and $\widetilde f|_{\widetilde M}$ has saddles of different indices \cite{OsPo2024}, that is $C_{2}>0$ and $C_{1}>0$. There are two optimal possibilities: 1 source, 1 saddle of index 2, $l_2$ saddles of index 1, and $l_2$ sinks or 2 source, 1 saddle of index 2, $l_2 - 1$ saddles of index 1, and $l_2$ sinks, --- for the both possibilities a total number of points in non-wandering set of $\widetilde f|_{\widetilde M}$ is $2l_2+2$, so $f|_{M}$ has at least $l_2+2$ isolated periodic points.

\end{proof}

\section{Achievability of the estimates}\label{sec:realization}

Realizations of diffeomorphisms with a minimum number of trivial basic sets are given in this section, i.e. we will prove the second part of theorems \ref{theo:minimum_number} and \ref{theo:non-or_manifold}.
First of all, we will answer on a question: how to obtain an $\Omega$-stable cascade $f:M^3\to M^3$ with a set of expanding attractors of codimension 1 $\Lambda$ with $k_1\geqslant 0$ bunches of degree 1 and $k_2\geqslant 0$ bunches of degree 2 in total ($k_1+k_2>0$) and $\frac32 k_1+k_2$ periodic points outside of $\Lambda$.

Let $f$ be a diffeomorphism of considered class with the following properties:

\begin{itemize}
\item all bunches and isolated periodic points are fixed;
\item if $k_2>0$, than $M^+$ is connected and has $k_2$ boundary components, otherwise $M^+$ is empty;
\item if $k_1>0$, than each non-trivial attractor has 1-bunches and $M^-$ has $k_1/2$ connected components, each of which is homeomorphic to $\mathbb RP^2\times[-1,1]$.
\end{itemize}

Corresponding regular system $\widetilde f|_{\widetilde M^+}$ for the set $M^+$ realizing the minimum can be as in figure \ref{fig:ms-flower}. It has $k_2$ sinks, $k_2-1$ saddles, and $1$ source. 
\begin{figure}[h]
  \centering
  \includegraphics[width=.5\linewidth]{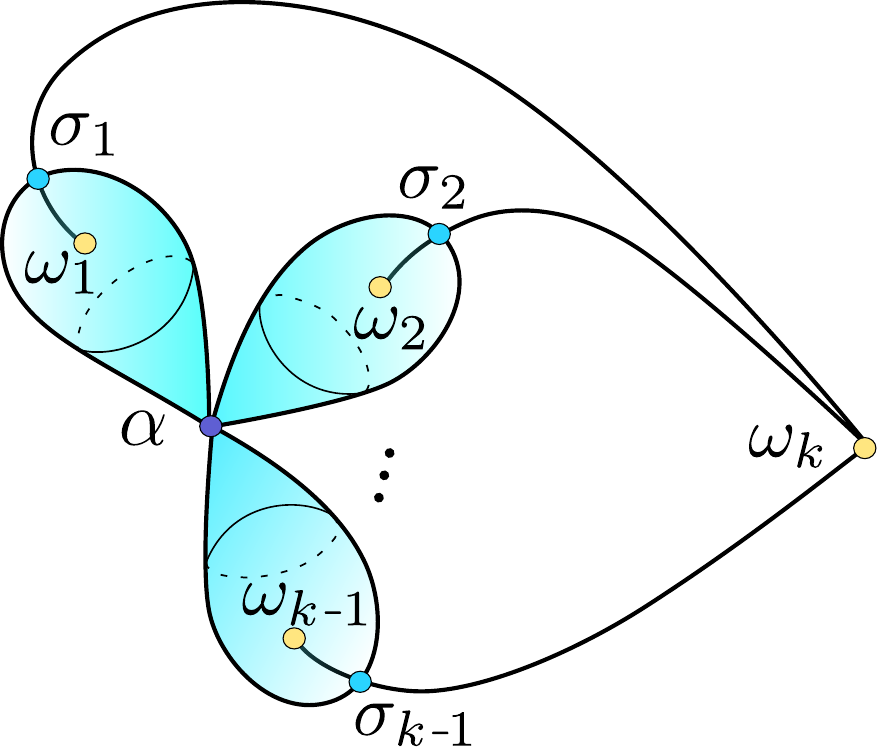}
  \caption{Morse-Smale system for $\mathcal M^+$, realizing low estimates}
  \label{fig:ms-flower}
\end{figure} 

If $k_1>0$, all bunches of degree 1 are divided into pairs in such a way that after gluing the cylinders $\mathbb RP^2\times[-1,1]$ to a trapping neighborhood of $\Lambda$ we will obtain a connected manifold $M^3$. Let the restriction $f|_{M}$ of the desired diffeomorphism $f$ on each connected component $M$ of $\mathcal M^-$ is topologically conjugated to a diffeomorphism $(g_1\times g_2)$, where $g_1:\mathbb RP^2\to\mathbb RP^2$ is on figure \ref{fig:system_on_rp2} and $g_2:\mathbb R\to\mathbb R$ such that $g_2(x)=2x$.
\begin{figure}[ht!]
  \centering
  \includegraphics[width=.4\linewidth]{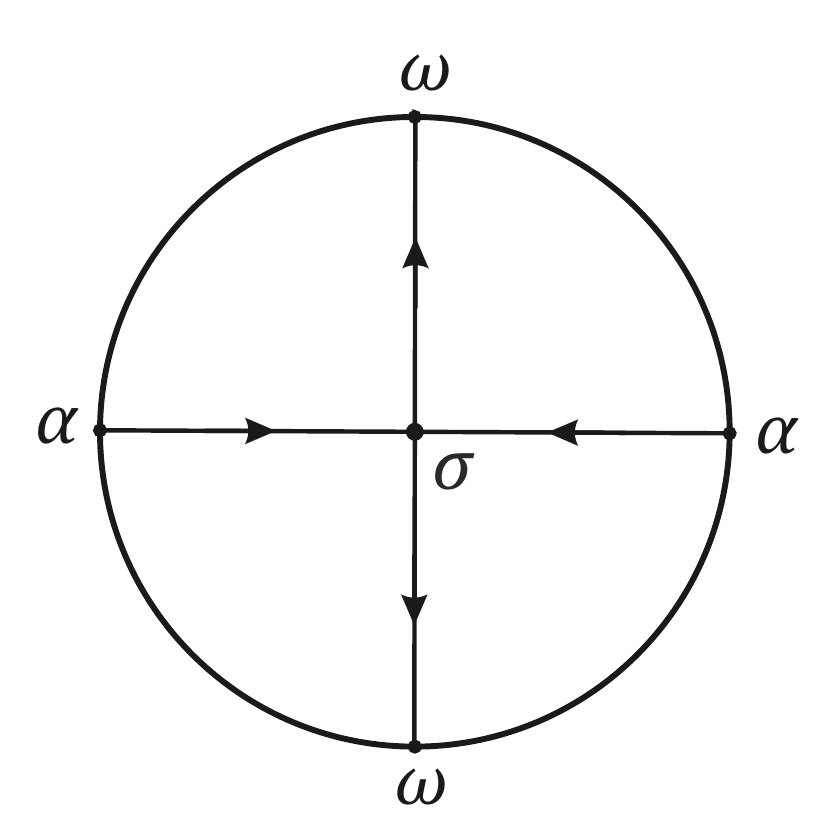}
  \caption{Morse-Smale system on $\mathbb RP^2$}
  \label{fig:system_on_rp2}
\end{figure} 

Achievability of the estimate from theorem \ref{theo:non-or_manifold} we will show with a diffeomorphism $f:M^3\to M^3$ with the following properties:

\begin{itemize}
\item $f$ has only 1 non-trivial attractor $\Lambda$, which is connected and has $k_2$ bunches of degree 2;
\item all bunches and isolated periodic points of $f$ is fixed;
\item a set $M^+$ consists of $k_2$ connected components.
\end{itemize}

Let a corresponded regular system $\widetilde f:\widetilde M^3\to\widetilde M^3$ be given on $\mathbb S^3\times\mathbb Z_{k_2-1}\sqcup \mathbb S^2\widetilde{\times}\mathbb S^1$, the dynamics on each 3-sphere be ``sink-source`` and on the $\mathbb S^2\widetilde{\times}\mathbb S^1$ be as on the figure \ref{fig:MS-non-orientable-sink}. Therefore $f:M^3\to M^3$ has exactly $k_2+2$: $k_2$ sources and 2 saddles of different indices, - isolated chain recurrent points and $M^3$ is non-orientable.

\begin{figure}[ht!]
  \centering
  \includegraphics[width=.5\linewidth]{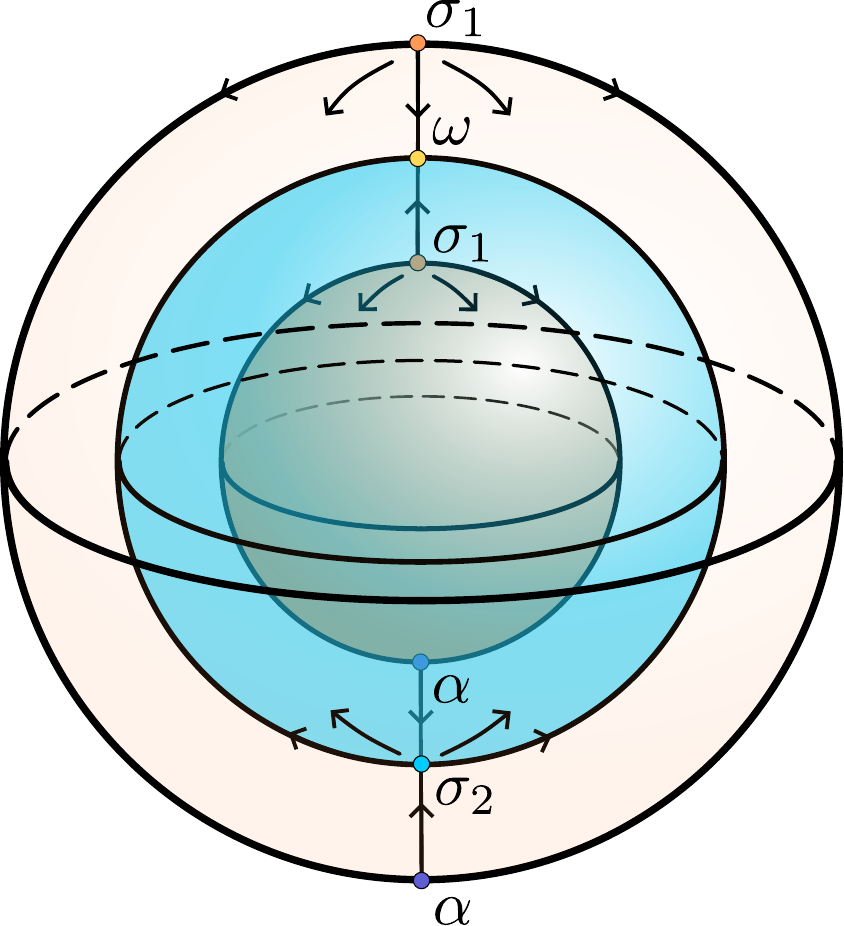}
  \caption{Morse-Smale system on $\mathbb S^2\widetilde{\times}\mathbb S^1$}
  \label{fig:MS-non-orientable-sink}
\end{figure}

\textit{Acknowledgments} This article is an output of a research project implemented as part of the Basic Research Program at the National Research University Higher School of Economics (HSE University).

\bibliographystyle{ieeetr}
\bibliography{biblio}

\providecommand{\noopsort}[1]{}\providecommand{\singleletter}[1]{#1}%
\begin{thebibliography}{10}

\bibitem{PaMelo}
J.~J. Palis and W.~De~Melo, {\em Geometric theory of dynamical systems: an
  introduction}.
\newblock Springer Science \& Business Media, 2012.

\bibitem{ShubStab}
M.~Shub, ``Stabilit{\'e} globale des syst{\`e}mes dynamiques,'' 1978.

\bibitem{SmaleOmega}
S.~Smale, ``The $\omega$-stability theorem, global analysis,'' in {\em Proc.
  Symp. Pure Math}, vol.~14, pp.~289--297, 1970.

\bibitem{FrankeHyp}
J.~E. Franke and J.~F. Selgrade, ``Hyperbolicity and chain recurrence,'' {\em
  Journal of Differential Equations}, vol.~26, no.~1, pp.~27--36, 1977.

\bibitem{Smale1967}
S.~Smale, ``Differentiable dynamical systems,'' {\em Bull. Amer. Math. Soc.},
  vol.~73, pp.~747--817, 11 1967.

\bibitem{Hirsch}
M.~W. Hirsch, {\em Differential topology}, vol.~33.
\newblock Springer Science \& Business Media, 2012.

\bibitem{Grines1975}
V.~Grines, ``On topological conjugacy of diffeomorphisms of a two-dimensional
  manifold onto one-dimensional orientable basic sets i,'' {\em Transactions of
  the Moscow Mathematical Society}, vol.~32, pp.~31--56, 1975.

\bibitem{Plykin1971}
R.~Plykin, ``The topology of basis sets for smale diffeomorphisms,'' {\em Math.
  USSR-Sb.}, vol.~13, pp.~301--312, 1971.

\bibitem{Plykin1984}
R.~V. Plykin, ``On the geometry of hyperbolic attractors of smooth cascades,''
  {\em Russian Mathematical Surveys}, vol.~39, no.~6, p.~85, 1984.

\bibitem{BaGrPo2023}
M.~K. Barinova, V.~Z. Grines, and O.~V. Pochinka, ``Dynamics of
  three-dimensional a-diffeomorphisms with two-dimensional attractors and
  repellers,'' {\em Journal of Difference Equations and Applications}, vol.~29,
  no.~9--12, pp.~1275--1286, 2023.

\bibitem{MeZhu2002}
E.~V. Zhuzhoma and V.~S. Medvedev, ``On non-orientable two-dimensional basic
  sets on 3-manifolds,'' {\em Sbornik: Mathematics}, vol.~193, no.~6,
  pp.~869--888, 2002.

\bibitem{BaPoYa2024}
M.~Barinova, O.~Pochinka, and E.~Yakovlev, ``On a structure of non-wandering
  set of an omega-stable 3-diffeomorphism possessing a hyperbolic attractor,''
  {\em Discrete and Continuous Dynamical Systems}, vol.~44, no.~1, pp.~1--17,
  2024.

\bibitem{GrLauPo2009}
V.~Grines, F.~Laudenbach, and O.~Pochinka, ``Self-indexing energy function for
  morse-smale diffeomorphisms on 3-manifolds,'' {\em Moscow Mathematical
  Journal}, vol.~9, no.~4, pp.~801--821, 2009.

\bibitem{OsPo2024}
E.~M. Osenkov and O.~V. Pochinka, ``MorseвСРsmale 3-diffeomorphisms with
  saddles of the same unstable manifold dimension,'' {\em Rus. J. Nonlin.
  Dyn.}, vol.~20, no.~1, pp.~167--178, 2024.

\bibitem{GMPZ-global}
V.~Z. Grines, E.~V. Zhuzhoma, V.~S. Medvedev, and O.~V. Pochinka, ``Global
  attractor and repeller of morse-smale diffeomorphisms,'' {\em Proceedings of
  the Steklov Institute of Mathematics}, vol.~271, no.~1, pp.~103--124, 2010.

\end{thebibliography}
\end{document}